\documentclass[12pt]{article} 
\usepackage[margin = 20mm]{geometry}
\usepackage{amsmath}
\usepackage{amsfonts}
\usepackage{mathrsfs}
\usepackage{amssymb}
\usepackage{amsthm}
\usepackage{cases}
\usepackage{blkarray}
\usepackage{enumerate}
\usepackage{tikz}
\usepackage{tkz-graph}
\usetikzlibrary{shapes.geometric}
\usetikzlibrary{babel}
\usepackage{pgfplots}
\pgfplotsset{compat=1.5}
\usepackage{tocloft}
\usepackage{tabularx}
\usepackage{setspace}
\usepackage{array,color,colortbl}   
\providecommand{\con}{\cellcolor{blue!15}}
\providecommand{\ctw}{\cellcolor{red!35}}
\providecommand{\cth}{\cellcolor{green!55}}
\usepackage{abstract}
\usepackage{nicematrix}
\usepackage{xcolor,colortbl}
\usepackage[ruled]{algorithm2e}
\usepackage[backend = biber, doi = false, url = false, isbn = false, maxbibnames=222]{biblatex}
\renewbibmacro{in:}{}
\newtheorem{thm}{Theorem}[section]
\newtheorem{cor}[thm]{Corollary}
\newtheorem{lem}[thm]{Lemma}

\theoremstyle{definition}

\numberwithin{equation}{section}
\def\eref#1{$(\ref{#1})$}
\def\sref#1{\S$\ref{#1}$}
\def\lref#1{Lemma~$\ref{#1}$}
\def\tref#1{Theorem~$\ref{#1}$}
\def\fref#1{Figure~$\ref{#1}$}

\renewcommand{\geq}{\geqslant}
\renewcommand{\leq}{\leqslant}
\renewcommand{\ge}{\geqslant}
\renewcommand{\le}{\leqslant}
\renewcommand{\emptyset}{\varnothing}
\renewcommand{\L}{\mathbf{L}}

\newcommand{\C}{\mathcal{C}}
\newcommand{\E}{\mathbb{E}}

\renewcommand{\epsilon}{\varepsilon}

\makeatletter
\g@addto@macro\bfseries{\boldmath}
\makeatother

\title{Subsquares in random Latin rectangles}

\author{Jack Allsop \ \ Ian M. Wanless\\
	\small School of Mathematics\\[-0.5ex]
	\small Monash University\\[-0.5ex]
	\small Vic 3800, Australia\\
	\small\tt jack.allsop@monash.edu \ \ ian.wanless@monash.edu}

\date{}

\begin{document}
	
	\maketitle
	
	\begin{abstract}
	  Suppose that $k$ is a function of $n$ and $n\to\infty$.  We show
	  that with probability $1-O(1/n)$, a uniformly random $k\times n$
	  Latin rectangle contains no proper Latin subsquare of order $4$ or
	  more, proving a conjecture of Divoux, Kelly, Kennedy and Sidhu. We
	  also show that the expected number of subsquares of order 3 is
	  bounded and find that the expected number of subsquares of order
          2 is $\binom k2(1/2+o(1))$ for all $k\le n$.
	\end{abstract}
	
	\section{Introduction}\label{s:intro}
	
	Throughout this paper, $n$ and $k$ will denote positive integers with $k \leq n$ and all asymptotics are as $n \to \infty$, with $k$ some function of $n$. All probability distributions will be discrete and uniform, with $\Pr(\cdot)$ denoting probability.
	
	A $k \times n$ \emph{Latin rectangle} $L$ is a $k \times n$ matrix on $n$ symbols, each of which occurs at most once in each row and column. If $k=n$ then $L$ is a \emph{Latin square} of order $n$. A \emph{subsquare} of $L$ is a submatrix of $L$ which is itself a Latin square. Every $k \times n$ Latin rectangle has $kn$ subsquares of order $1$, and every Latin square of order $n$ has a single subsquare of order $n$. A \emph{proper} subsquare of $L$ is a subsquare which has order in the set $\{2, 3, \ldots, n-1\}$. The order of any proper subsquare of $L$ is at most $\min\{k, n/2\}$. An \emph{intercalate} is a subsquare of order $2$. 
	
	Let $m$ be an integer satisfying $2 \leq m \leq n/2$ and
	define $\E_m(k,n)$ to be the expected number of subsquares of
	order $m$ in a random $k \times n$ Latin rectangle.  The main
	result of this paper is the following theorem.
	
	\begin{thm}\label{t:main}
		If $m,k$ are integer functions of $n$ satisfying $4\le
		m\le\min\{k,n/2\}$ then $\E_m(k,n)=O(n^{-2})$ as $n\to\infty$.
	\end{thm}
	
	We also show that $\E_3(k,n)=O(k/n)$. An immediate corollary of
	\tref{t:main} is the following result, proving a conjecture of
	Divoux, Kelly, Kennedy and Sidhu~\cite{subsqrandom}, which generalises
	a conjecture by McKay and Wanless~\cite{manysubsq}.
	
	\begin{cor}\label{c:main}
		With probability $1-O(1/n)$, a random $k \times n$ Latin rectangle has
		no proper subsquare of order $4$ or more.
	\end{cor}

	The structure of this paper is as follows. In \sref{s:back} we discuss
	some of the literature regarding subsquares in random Latin squares
	and rectangles, as well as giving some definitions which we require to
	prove \tref{t:main}. We prove \tref{t:main} in \sref{s:main}, and in
	\sref{s:subsq3} we show that $\E_3(k, n)=O(k/n)$ and
        $\mathbb{E}_2(k,n)=(1/2+o(1))\binom{k}{2}$ for all $k\le n$.
	
	\section{Background}\label{s:back}
	
	A seminal paper by Godsil and McKay \cite{GM91} used switching
	to count substructures in random Latin rectangles.  A special
	case of their Theorem 4.7 is that if $k \leq n/6$ then the
	probability that a random $k\times n$ Latin rectangle contains
	a specific subsquare of order $m$ is $\exp(-m^2\log
	n+O(m^2k/n))$.  There are less than
	$\binom{k}{m}\binom{n}{m}^2\le(en/m)^{3m}$ ways to choose the
	rows, columns and symbols for a potential subsquare of order
	$m$ and less than $m^{m^2}$ Latin squares with those rows,
	columns and symbols. It follows that if $k \leq n/6$ then
	\[\E_m(k,n)\le\exp((m^2-3m)\log(m/n)+O(m)+O(m^2k/n))\]
	and, using precise values rather than approximations,
	\begin{equation}\label{e:E2smallk}
		\E_2(k,n)=\frac12\binom{k}2\exp(O(k/n)).
	\end{equation}
	In particular, if $k=o(n)$ then a random $k\times n$ Latin rectangle
	almost surely has no subsquare of order 4 or larger.

	The first researchers to study subsquares of random Latin
	squares were McKay and Wanless~\cite{manysubsq}. They provided
	some estimates for the number of intercalates and conjectured that
	asymptotically almost surely there will be no proper subsquare
	of order $4$ or more. They also conjectured that
	$\E_3(n,n)=1/18+o(1)$. They proved that with probability
	$1-o(1)$, a random Latin square of order $n$ has no subsquare
	of order $n/2$.  Estimates of the number of intercalates were
	subsequently improved in \cite{cycstrucrandom,KSS,KSSS,KS},
	leading to the following result.
	
	\begin{thm}\label{t:intercs}
	  Let $\mathbf{N}$ denote the number of intercalates in a random Latin
	  square of order $n$.
	  \begin{itemize}
	  \item The expected value of $\mathbf{N}$ is $(1+o(1))n^2/4$,
	  \item $\Pr(\mathbf{N} \leq (1-\delta)n^2/4) = \exp(-\Theta(n^2))$ for every $\delta \in (0, 1]$,
	  \item $\Pr(\mathbf{N} \geq (1+\delta)n^2/4) = \exp(-\Theta(n^{4/3}\log n))$ for every $\delta>0$.
	  \end{itemize} 
	\end{thm}
	
	After two decades of no progress on the McKay-Wanless
        conjectures, two research groups recently managed to settle
        the case of large subsquares. Divoux, Kelly, Kennedy and
        Sidhu~\cite{subsqrandom} proved that there is a constant $K>0$
        such that asymptotically almost all Latin squares of order $n$
        contain no proper subsquare of order $K(n\log n)^{1/2}$ or
        more.  They also showed that for $k\le(1/2-o(1))n$, a random
        $k\times n$ Latin rectangle contains no proper subsquare of
        order $4$ or more with probability $1-o(1)$, and that
        $\E_3(k,n)=(1+o(1))\binom{k}{3}/(3n^3)$ and $\E_2(k,n)=\binom
        k2(1/2+o(1))$. They conjectured that these results hold for
        all $k\le n$, generalising the conjectures made by McKay and
        Wanless \cite{manysubsq}.
	
	Independently, Gill, Mammoliti and Wanless~\cite{canonicallabel}
	showed that for any $\epsilon>0$, asymptotically almost all Latin
	squares contain no proper subsquare of order
	$n^{1/2}\log^{1/2+\epsilon} n$ or more. They also proved that their
	result implies that isomorphism for Latin squares can be tested in
	average case polynomial time.

	\medskip
	
	We now give some definitions which we will need to prove
	\tref{t:main}. Let $u$ and $v$ be positive integers. We denote the set
	$\{1, 2, \ldots, u\}$ by $[u]$. Let $A$ be a $u \times v$ matrix. We
	will index the rows of $A$ by $[u]$ and the columns of $A$ by
	$[v]$. This convention will be adopted throughout this paper, unless
	stated otherwise. A pair $(i, j) \in [u] \times [v]$ is a \emph{cell}
	of $A$ and a triple $(i, j, A_{i, j})$ is an \emph{entry} of $A$. Let
	$I \subseteq [u]$ and $J \subseteq [v]$. We denote by $A[I,J]$ the
	submatrix of $A$ induced by the rows in $I$ and the columns in $J$.
	
	A $k \times n$ \emph{partial Latin rectangle} $P$ is a $k \times n$
	matrix such that each cell is either empty or contains one of $n$
	symbols, such that no symbol occurs more than once in each row and
	column.
	A \emph{subrectangle} of $P$ is a submatrix of $P$ which is itself a
	Latin rectangle. We will equivalently view $P$ as the set of triples
	$(r,c, s)$ where $(r,c)$ is a non-empty cell of $P$ and $s =
	P_{r,c}$. This allows us to use set notation such as $(1, 2, 3) \in
	P$, which means that $P_{1, 2}=3$. Throughout this paper, unless
	otherwise stated, a $k \times n$ partial Latin rectangle will have
	symbol set $[n]$. The convention that a Latin square of order $n$ has
	row indices, column indices and symbol set $[n]$ will be broken when
	dealing with subsquares.
	
	We will prove \tref{t:main} using the switching method. This method is behind many of the results on random Latin squares and rectangles including \cite{cycstrucrandom, subsqrandom, canonicallabel, GM91, KSS, KSSS, KS, manysubsq} and has been widely used for many other combinatorial structures. Our switching procedure will be cycle switching and incomplete cycle switching, which we define now. 
	
	Let $P$ be a $k \times n$ partial Latin rectangle. A \emph{row cycle} of $P$ is a $2 \times \ell$ subrectangle, for some $\ell \leq n$, which is minimal in the sense that it does not contain any $2 \times \ell'$ subrectangle with $\ell' < \ell$. Let $Q$ be a row cycle of $P$ which hits rows $i$, $j$ and the columns in $C$, for some $\{i, j\} \subseteq [k]$ and $C \subseteq [n]$. Since $Q$ is uniquely determined by rows $i$, $j$ and a single column in $C$, we denote $Q$ by $\rho(i, j, c)$, where $c$ is any element of $C$.	
	Row cycles give us a way of perturbing partial Latin rectangles to create new ones. We can define a partial Latin rectangle $P'$ by
	\[
	P'_{x, y} = \begin{cases}
		P_{i, y} & \text{if } x=j \text{ and } y \in C, \\
		P_{j, y} & \text{if } x=i \text{ and } y \in C, \\
		P_{x, y} & \text{otherwise}.
	\end{cases}
	\]
	We say that $P'$ has been obtained from $P$ by switching on $\rho(i, j, c)$. 
	
	A \emph{column cycle of length $\ell$} is an $\ell \times 2$ submatrix
	of $P$ that contains exactly $\ell$ symbols and no empty cells and
	does not contain any smaller submatrix with these properties. The
	entries in a column cycle form a row cycle when $P$ is transposed.
	Suppose that $\C$ is an $\ell \times 2$ submatrix of $P$ that does not
	contain a column cycle or any empty cells and which hits columns $i$
	and $j$ and rows $R$ of $P$. Then $\C$ is an \emph{incomplete column
		cycle of length $\ell$} if there are unique rows $r,r'\in R$ such
	that $P_{r,i}$ does not occur in column $j$ of $P$ and $P_{r',j}$ does
	not appear in column $i$ of $P$.
	See \fref{f:colpath} for an example of an incomplete column cycle.
	
	\begin{figure}
		\[\left[
		\begin{matrix}
			2&\cellcolor{blue!15}3&5&6&\cellcolor{blue!15}1&4&7\\
			5&\cellcolor{blue!15}4&7&2&\cellcolor{blue!15}3&6&1\\
			4&5&1&3&7&2&6\\
			6&\cellcolor{blue!15}1&3&4&\cellcolor{blue!15}2&7&5	
		\end{matrix}
		\right]
		\]
		\caption{\label{f:colpath}An incomplete column cycle in a $4 \times 7$
			Latin rectangle}
	\end{figure}
	
	Let $\{i, j\} \subseteq [n]$ with $i \neq j$ and let $r\in[k]$. There
	is either a unique column cycle or a unique incomplete column cycle
	(but not both) which hits columns $i$ and $j$ and row $r$. Denote this
	column cycle or incomplete column cycle by $\sigma(i, j, r)$, and let
	$R$ be the set of rows which it hits. The substructure $\sigma(i,j,r)$
	gives us a way of perturbing $P$ to create a new Latin
	rectangle. Define a partial Latin rectangle $P'$ by
	\[
	P'_{x, y} = \begin{cases}
		P_{x, i} & \text{if } x \in R \text{ and } y = j, \\
		P_{x, j} & \text{if } x \in R \text{ and } y = i, \\
		P_{x, y} & \text{otherwise}.
	\end{cases}
	\]
	We say that $P'$ has been obtained from $P$ by switching on $\sigma(i, j, r)$.
	
	There are also symbol cycles and incomplete symbol cycles, which are,
	respectively, the image of column cycles and incomplete column cycles
	under the map which replaces each entry $(r,c,s)$ by the entry
	$(r,s,c)$.  Symbol cycles and incomplete symbol cycles can also be
	switched in an analogous way.  See~\cite{cycswitch} for a study of
	switching on row, column and symbol cycles.

	Let $L$ be a $k \times n$ Latin rectangle. Let $\{i, j\} \subseteq [k]$ with $i \neq j$. The permutation mapping row $i$ of $L$ to row $j$, denoted by $\tau_{i, j}$, is defined by $\tau_{i, j}(L_{i, \ell}) = L_{j, \ell}$ for every $\ell \in [n]$. Such permutations are called \emph{row permutations} of $L$. Let $\rho$ be a cycle in $\tau_{i, j}$ and in row $i$ let $C$ be the set of columns containing the symbols involved in $\rho$. Then the set of entries 
	\[
	\{(i, c, L_{i, c}), (j, c, L_{j, c}) : c \in C\}
	\]
	is the row cycle $\rho(i, j, c)$ of $L$, where $c$ is any element of $C$. Conversely, every row cycle of $L$ corresponds to a cycle of a row permutation of $L$. 
	
	Now suppose that $\{i, j\} \subseteq [n]$ with $i \neq j$. The partial
	permutation mapping column $i$ to column $j$, denoted by $\kappa_{i,j}$,
	is defined by $\kappa_{i, j}(L_{\ell, i}) = L_{\ell, j}$ for
	every $\ell \in [k]$. A cycle in $\kappa_{i, j}$ corresponds to a
	column cycle of $L$ hitting columns $i$ and $j$. Say that a list
	$[x_1, x_2, \ldots, x_u]$ is an \emph{incomplete cycle} of $\kappa_{i,j}$
	if $\kappa_{i, j}(x_i) = x_{i+1}$ for every $i \in [u-1]$ and
	$\kappa_{i, j}(x_u)$ and $\kappa_{i, j}^{-1}(x_1)$ are undefined. An
	incomplete cycle of $\kappa_{i, j}$ corresponds to an incomplete
	column cycle of $L$ hitting columns $i$ and $j$.
	
	\section{Proof of main theorem}\label{s:main}
	
	Throughout this section $m = m(n)$ will be a positive integer
	satisfying $2 \leq m \leq \min\{k, n/2\}$.
	This section is split into three subsections. In \sref{ss:subsqbound}
	and \sref{ss:subsq4bound} we prove two different bounds on the
	probability that a random $k \times n$ Latin rectangle contains a
	subsquare of order $m$ in a specific selection of rows and columns,
	and on a specific symbol set. These bounds are effective for different
	but overlapping ranges of $m$. In \sref{ss:combine} we combine them to prove \tref{t:main}, using the bound from \sref{ss:subsqbound} for $m\ge6$ and the
	bound from \sref{ss:subsq4bound} for $m\in\{4,5\}$.
	
	\subsection{Bounding probability of large subsquares}\label{ss:subsqbound}
	
	The main result of this subsection is the following theorem, which
	gives a nontrivial bound for all~$m$. However, for $m=4$ it is not
	strong enough to imply any result in the direction of
	\tref{t:main} and for $m=5$ it is not enough to imply the full strength of
	\tref{t:main}.  The bound we prove in the next subsection will be
	better whenever $2<m=O(1)$.
	
	\begin{thm}\label{t:subsqbound}
		Let $R \subseteq [k]$ be of cardinality $m$, let $C$ and $S$
		be subsets of $[n]$ of cardinality $m$, and let $\L$ be a
		random $k \times n$ Latin rectangle. The probability that
		$\L[R,C]$ is a Latin square of order $m$ with symbol set $S$
		is at most
		\begin{equation*}\label{e:ub2}
			\frac{kn(k+1-m)(n+1-m)}{m^2\binom{n}{m}^3\binom{k}{m}}.
		\end{equation*}
	\end{thm}
	
	Without loss of generality we may assume that $R = C = S = [m]$ in \tref{t:subsqbound}. Let $T_0=\emptyset$ and for $i \in [m^2]$ define $T_i$ to be the set of all pairs $(r,c)\in R\times C$ such that $(c-1)m+r \leq i$.
	We will consider the sets $T_i$ to be sets of cells of $k \times n$ Latin rectangles. So if $a \in \{0, 1, \ldots, m-1\}$, $b \in [m]$ and $L$ is a $k \times n$ Latin rectangle then $T_{am+b}$ consists of all cells of $L$ in the first $m$ rows and $a$ columns, or in the first $b$ rows and the $(a+1)$-st column. Define $t_i$ to be the cell in $T_i \setminus T_{i-1}$. Define $\Delta_i$ to be the set of $k \times n$ Latin rectangles such that the symbol in cell $(r, c)$ is an element of $[m]$, for every $(r, c) \in T_i$. Note that $\Delta_0$ is the set of all $k \times n$ Latin rectangles. 
	
	Let $\L$ be a random $k \times n$ Latin rectangle. The probability that $\L[[m], [m]]$ is a Latin square of order $m$ with symbol set $[m]$ is $\Pr(\L \in \Delta_{m^2})$. By the chain rule of probability,
	\begin{equation}\label{e:cr}
		\Pr(\L \in \Delta_{m^2}) = \prod_{i=1}^{m^2} \Pr(\L \in \Delta_i \mid  \L \in \Delta_{i-1}).
	\end{equation}
	Similar to the approach of Divoux, Kelly, Kennedy and Sidhu~\cite{subsqrandom}, our approach to proving \tref{t:subsqbound} is to bound the terms 
	\begin{equation}\label{e:term}
		\Pr(\L \in \Delta_i \mid  \L \in \Delta_{i-1}).
	\end{equation}
	However we will only provide a non-trivial bound on \eref{e:term} when the cell $t_i = (r, c)$ satisfies $\{1, m\} \cap \{r, c\} \neq \emptyset$. That is, we only bound \eref{e:term} when $t_i$ lies in the first row, first column, $m$-th row or $m$-th column. Slightly surprisingly, this turns out to be enough to derive our result.  Consider when $k=4$, $n=6$, $m=3$ and $i=5$. The Latin rectangle
	\[
	L = \left[\;
	\begin{matrix}		
		\con1&\con2&6&5&4&3 \\
		\con2&\con3&5&4&6&1 \\
		\con3&4&1&2&5&6 \\
		4&6&2&1&3&5 \\
	\end{matrix}\;
	\right]
	\]
	is a member of $\Delta_i$, due to the highlighted entries. Suppose that we want to use column cycle or incomplete column cycle switching to get from $L$ to a rectangle in $\Delta_{i-1} \setminus \Delta_i$. Clearly we must switch on $\sigma(2, c, 2)$ for some $c \in \{3, 4, 5, 6\}$. If we take $c=5$ it is clear that the resulting rectangle will lie in $\Delta_{i-1} \setminus \Delta_i$, as desired. However, if we take $c=4$ the resulting rectangle will not be a member of $\Delta_{i-1}$, since cell $(1, 2)$ will no longer contain a symbol in $[3]$. So to estimate the number of switches from some rectangle $L' \in \Delta_i$ to a rectangle in $\Delta_{i-1} \setminus \Delta_i$, we must estimate the number of columns $c$ of $L'$ for which $\sigma(2, c, 2)$ does not hit row $1$. When dealing with general $n$, $m$ and $i$ with $t_i = (r, c)$, to determine the number of switches from a rectangle $L' \in \Delta_i$ to a rectangle in $\Delta_{i-1} \setminus \Delta_i$, we will estimate the number of columns $c'$ of $L'$ for which $\sigma(c,c',r)$ does not hit any row $r'<r$. This seems like a difficult task in general. However, when $r=1$ this difficulty does not arise, and it is clear that we can switch $L'$ on $\sigma(c, c', 1)$ for any choice of $c' \in [n] \setminus [c]$ such that $L'_{r, c'} \notin [m]$ and obtain a rectangle in $\Delta_{i-1} \setminus \Delta_i$. This explains why we can provide non-trivial bounds on \eref{e:term} when $t_i$ is in the first row. By using symbol cycle and incomplete symbol cycle switching instead of column cycle and incomplete column cycle switching and applying the same logic, we see why we can obtain non-trivial bounds on \eref{e:term} when $t_i$ is in the first column. Note that we could also use row cycle switching to achieve this, but we get stronger bounds when we use symbol cycle and incomplete symbol cycle switching. 
	When $t_i = (r, m)$ is in the $m$-th column and $c \in [n] \setminus [m]$ we can estimate the number of rectangles in $\Delta_i$ for which $\sigma(m, c, r)$ does not hit any row $r' < r$ (see \lref{l:prelim2}). We can obtain similar results for when $t_i$ is in the $m$-th row. This allows us to provide non-trivial bounds on \eref{e:term} in these cases.
	
	\medskip
	
	The following lemma allows us to bound \eref{e:term} when $t_i$ is in
	the first row or column.
	
	\begin{lem}\label{l:setnrc}
		Let $i\in\{m+1-\alpha,m^2+1-\alpha m\}$ for some $\alpha\in[m]$.
		For a random $k \times n$ Latin rectangle $\L$,
		\[
		\Pr(\L \in \Delta_i \mid  \L \in \Delta_{i-1})=\frac{\alpha}{n+\alpha-m}.
		\]
	\end{lem}
	\begin{proof}
		First suppose that $i = m^2+1-\alpha m$ for some $\alpha \in [m]$. We use column cycle and incomplete column cycle switching to find the relative sizes of $\Delta_i$ and $\Delta_{i-1}$. Let $L \in \Delta_i$. Let $c' \in [n] \setminus [m-\alpha+1]$ be such that $L_{1, c'} \not\in [m]$. Note that there are exactly $n-m$ choices for $c'$. Let $L'$ be obtained from $L$ by switching on $\sigma(m-\alpha+1, c', 1)$. Then $L'_{1, m-\alpha+1} \not\in [m]$ and $L'_{r, c} = L_{r, c}$ for every cell $(r, c) \in T_{i-1}$. Thus $L' \in \Delta_{i-1} \setminus \Delta_i$. It follows that there are exactly $(n-m)|\Delta_i|$ switches from a rectangle in $\Delta_i$ to a rectangle in $\Delta_{i-1} \setminus \Delta_i$.
		
		Reversing the switching process, consider $L'\in \Delta_{i-1} \setminus \Delta_i$ that is obtained from some $L\in\Delta_i$ by switching on some $\sigma(m-\alpha+1, c', 1)$. Here $c'$ must be a column in $[n]\setminus [m-\alpha+1]$ such that $L'_{1, c'} \in [m]$.
		Since $L'_{1,\ell}\in[m]$ for $\ell\in[m-\alpha]$ and $L'_{1, m-\alpha+1}\notin[m]$, there are exactly $\alpha$ choices for $c'$. Thus there are exactly $\alpha|\Delta_{i-1} \setminus \Delta_i|$ switches from a rectangle in $\Delta_{i-1} \setminus \Delta_i$ to a rectangle in $\Delta_i$.
		
		We have shown that $(n-m)|\Delta_i| = \alpha|\Delta_{i-1} \setminus \Delta_i|$. Therefore,
		\[
		\Pr(\L \in \Delta_i \mid  \L \in \Delta_{i-1}) = \frac{|\Delta_i|}{|\Delta_i|+|\Delta_{i-1} \setminus \Delta_i|} = \frac{\alpha}{n+\alpha-m},
		\]
		as required. To deal with the case where $i=m+1-\alpha$ we use symbol cycle and incomplete symbol cycle switching instead of column cycle and incomplete column cycle switching. Crucially, no two cells in $T_{m+1-\alpha}$ contain the same symbol, since we look at the first column before we look at any other cells.
	\end{proof}
	
	Our next task is to bound \eref{e:term} when $t_i$ is in the $m$-th row or column. To do that we need to prove the preliminary results \lref{l:prelim} and \lref{l:prelim2} below. We also need the following definitions. For a set $T \subseteq [m]^2$ of cells define $\Gamma(T)$ to be the set of $k \times n$ partial Latin rectangles $P$ such that the set of non-empty cells of $P$ is exactly $T$ and every symbol of $P$ is in $[m]$. Also define 
	\begin{equation}\label{e:Gammastar}
		\Gamma^*(T) = \bigcup_{T' \subseteq T} \Gamma(T').
	\end{equation}
	
	\begin{lem}\label{l:prelim}
		Let $P \in \Gamma^*([m]^2)$, and let $(r, c)$ be a non-empty cell of
		$P$.  Also let $r' \in [k]\setminus [m]$. Define
		$C'=\{c,c_1,c_2,\ldots,c_\ell\}\subseteq[m]$ to be the set of columns
		$c'$ for which cell $(r,c')$ of $P$ is non-empty.
		Suppose that
		the entries of $P$ in columns $c$ and $c_i$ form a union of column
		cycles, for every $i \in [\ell]$.  
		Let $A$ be the set of $k \times n$ Latin rectangles containing $P$, and
		let $X\subseteq A$ be the set of rectangles in $A$ such that
		the row cycle $\rho(r, r', c)$ does not hit any column in $C'\setminus\{c\}$.
		Let $\L$ be a random $k \times n$ Latin rectangle. Then
		\[
		\Pr(\L \in X\mid  \L \in A)=\frac1{\ell+1}.
		\]
	\end{lem}
	
	\begin{proof}
		Let $U \subseteq [m]$ be the set of rows containing a non-empty cell
		of $P$ in column $c$.  Define $Y = A \setminus X$ and a map
		$\phi:Y \to X$ as follows. Let $L \in Y$ and let $s = L_{r, c}$. Write the
		cycle of $\tau_{r, r'}$ containing $s$ as $(s, x_1, x_2, \ldots,x_u)$.
		Let $v \leq u$ be maximal such that $x_v = L_{r, c_j}$
		for some $j \in [\ell]$ and let $c' = c_j$ for this particular
		$j$. Note that $v$ exists since $L \in Y$. Also note that $v<u$
		since if $v = u$ this would imply that $(r', c', s) \in L$,
		contradicting the fact that the entries of $P$ in columns $c$ and
		$c'$ form a union of column cycles. Define $\phi(L)$ by swapping the
		symbols in cells $(r'', c)$ and $(r'', c')$ for every $r'' \in
		[k]\setminus U$. The fact that the entries of $P$ in columns $c$ and
		$c'$ form a union of column cycles implies that $\phi(L)$ is
		reached from $L$ by switching one or more column cycles or incomplete column cycles. 
		Hence, $\phi(L)$ is indeed a Latin rectangle and $\phi(L) \in A$.  The
		cycle of $\tau_{r, r'}$ of $\phi(L)$ containing $s$ is
		$(s,x_{v+1}, \ldots, x_u)$, meaning that $\phi(L) \in X$.
		
		Let $L'\in X$. We now argue that there are exactly $\ell$ rectangles $L\in
		Y$ such that $\phi(L)=L'$. Note that $L'$ is obtained from any such
		$L$ by swapping all symbols in cells $(r'', c)$ and $(r'', c_w)$
		for some $w \in [\ell]$, and all $r'' \in [k] \setminus U$. It is
		immediate that there are at most $\ell$ possible rectangles $L$ such that
		$\phi(L)=L'$. We need to show that each of the $\ell$ possibilities is
		realised.  The fact that $L'\in X$ ensures that $\tau_{r,r'}$ of $L'$
		contains two separate cycles $(\alpha_1,\alpha_2,\ldots,\alpha_a)$ and
		$(\beta_1,\beta_2,\ldots,\beta_b)$ where $\alpha_1=L'_{r,c}$ and
		$\beta_1=L'_{r,c_w}$. Swapping the contents of
		cells $(r'', c)$ and $(r'', c_w)$
		for all $r'' \in [k] \setminus U$ produces a Latin rectangle in which
		$\tau_{r,r'}$ contains the cycle
		$(\alpha_1,\beta_2,\ldots,\beta_b,\beta_1,\alpha_2,\ldots,\alpha_a)$, which hits
		both columns $c$ and $c_w$. By definition
		such a rectangle is in $Y$. 
		It follows that $|Y|=\ell|X|$ and
		\[
		\Pr\left(\L \in X\mid\L \in A\right)=\frac{|X|}{|X|+|Y|}=\frac1{\ell+1},
		\]
		as required.
	\end{proof}

	\begin{lem}\label{l:prelim2}
		Let $P \in \Gamma^*([m]^2)$, and let $(r, c)$ be a non-empty cell of
		$P$.  Also let $c' \in [n]\setminus [m]$. Define
		$R'=\{r,r_1,r_2,\ldots,r_\ell\}\subseteq[m]$ to be the set of rows
		$r'$ for which cell $(r',c)$ of $P$ is non-empty.
		Suppose that
		the entries of $P$ in rows $r$ and $r_i$ form a union of row
		cycles, for every $i \in [\ell]$.  
		Let $A$ be the set of $k \times n$ Latin rectangles containing $P$, and
		let $X\subseteq A$ be the set of rectangles in $A$ such that
		$\sigma(c, c', r)$ does not hit any row in $R'\setminus\{r\}$.
		Let $\L$ be a random $k \times n$ Latin rectangle Then
		\[
		\Pr(\L \in X\mid  \L \in A) \geq \frac1{\ell+1}.
		\]
	\end{lem}
	
	\begin{proof}
		The proof is similar to the proof of \lref{l:prelim}. Let $U \subseteq [m]$ be the set of columns containing a non-empty cell of $P$ in row $r$. Define $Y = A \setminus X$ and a map $\phi : Y \to X$ as follows. Let $L \in Y$ and let $s = L_{r, c}$. First suppose that $\sigma(c, c', r)$ is a column cycle. Write the cycle of $\kappa_{c, c'}$ containing $s$ as $(s, x_1, x_2, \ldots, x_u)$. Let $v < u$ be maximal such that $x_v = L_{r_j, c}$ for some $j \in [\ell]$ and let $r' = r_j$ for this particular $j$. By the same arguments as in the proof of \lref{l:prelim} we know that $v$ exists. Define $\phi(L)$ by swapping the symbols in cells $(r, c'')$ and $(r', c'')$ for every $c \in [n] \setminus U$. The fact that the entries of $P$ in rows $r$ and $r'$ form a union of row cycles implies that $\phi(L)$ is reached from $L$ by switching on one or more row cycles. Hence, $\phi(L)$ is indeed a Latin rectangle and $\phi(L) \in A$. The cycle of $\kappa_{c, c'}$ of $\phi(L)$ containing symbol $s$ is $(s, x_{v+1}, \ldots, x_u)$ meaning that $\phi(L) \in X$. Now suppose that $\sigma(c, c', r)$ is an incomplete column cycle. Write the incomplete cycle of $\kappa_{c, c'}$ as $[y_1, \ldots, y_{u'}, s, x_1, \ldots, x_u]$. We now consider two cases. First, suppose that there is some $v \in [u']$ such that $L_{r_j, c} = y_v$ for some $j \in [\ell]$. Let such a $v$ be maximal and let $r' = r_j$ for this value of $j$. If $v = u'$ then $L$ must contain the entry $(r', c', s)$, contradicting the fact that the entries of $P$ in rows $r$ and $r'$ form a union of row cycles. So $v < u'$. Let $\phi(L)$ be obtained from $L$ by swapping the contents in cells $(r, c'')$ and $(r', c'')$ for every $c'' \in [n] \setminus U$. The fact that the entries of $P$ in rows $r$ and	$r'$ form a union of row cycles implies that $\phi(L)$ is reached from $L$ by switching one or more row cycles. Hence, $\phi(L)$ is indeed a Latin rectangle and $\phi(L) \in A$. In $\phi(L)$, $\sigma(c, c', r)$ is a cycle and the cycle of $\kappa_{c, c'}$ of $\phi(L)$ containing $s$ is $(s, y_{v+1}, \ldots, y_{u'})$. Thus, $\phi(L) \in X$. Finally, suppose that no such $v$ exists. Since $L \in Y$ there is some $v' < u$ such that $L_{r_j, c} = x_{v'}$. Let such a $v'$ be maximal and let $r' = r_j$ for this value of $j$. Let $\phi(L)$ be obtained from $L$ by swapping the contents in cells $(r, c'')$ and $(r', c'')$ for every $c'' \in [n] \setminus U$. Then $\phi(L) \in A$ and the incomplete cycle of $\kappa_{c, c'}$ containing $s$ is $[y_1, \ldots, y_{u'}, s, x_{v'+1}, \ldots, x_u]$. Therefore, $\phi(L) \in X$.
		
		Let $L' \in X$. We now argue that there are at most $\ell$ rectangles
		$L\in Y$ such that $\phi(L)=L'$. This is true since $L'$ is obtained
		from any such $L$ by swapping all symbols in cells $(r, c'')$ and
		$(r_w, c'')$ for some $w \in [\ell]$, and all $c'' \in [n] \setminus
		U$.  It follows that $|Y| \leq \ell|X|$ and
		\[
		\Pr\left(\L \in X\mid\L \in A\right)=\frac{|X|}{|X|+|Y|} \geq \frac1{\ell+1},
		\]
		as required.
	\end{proof}

	We can now bound \eref{e:term} when $t_i$ is in the $m$-th row.
	
	\begin{lem}\label{l:setrowcycs}
		Let $i=jm$ where $j\in[m]$ and let $\L$ be a random $k \times n$
		Latin rectangle. Then
		\[
		\Pr(\L \in \Delta_i \mid  \L \in \Delta_{i-1}) \leq \frac{j}{k+j-m}.
		\]
	\end{lem}
	
	\begin{proof}
		We use row cycle switching to estimate the relative sizes of $\Delta_i$ and $\Delta_{i-1} \setminus \Delta_i$. For a partial Latin rectangle $P \in \Gamma(T_i)$ let $\mathcal{L}_P$ denote the set of $k \times n$ Latin rectangles containing $P$. So $\Delta_i$ is the disjoint union 
		\[
		\bigcup_{P \in \Gamma(T_i)} \mathcal{L}_P.
		\]
		
		Let $r' \in [k]\setminus [m]$. Note that if $L \in \Delta_i$ then $L_{m, j} \in [m]$ and $L_{r', j} \notin [m]$. Let $P \in \Gamma(T_i)$. By \lref{l:prelim} there are $|\mathcal{L}_P|/j$ rectangles $L \in \mathcal{L}_P$ such that switching $L$ on $\rho(m, r', j)$ yields a rectangle in $\Delta_{i-1} \setminus \Delta_i$. Thus the number of switches from a rectangle in $\Delta_i$ to one in $\Delta_{i-1} \setminus \Delta_i$ is $(k-m)|\Delta_i|/j$.
		
		Reversing the switching, consider $L' \in \Delta_{i-1} \setminus \Delta_i$. Note that $L'_{m, j} \notin [m]$ and there is at most one $r' \in [k]\setminus [m]$ such that $L'_{r', j} \in [m]$. Therefore there is at most one switch from $L'$ to a rectangle in $\Delta_i$.
		
		We have thus shown that $(k-m)|\Delta_i|/j \leq |\Delta_{i-1} \setminus \Delta_i|$ and so
		\[
		\Pr(\L \in \Delta_i \mid  \L \in \Delta_{i-1}) = \frac{|\Delta_i|}{|\Delta_i|+|\Delta_{i-1} \setminus \Delta_i|} \leq \frac j{k+j-m},
		\]
		as required.	
	\end{proof}
	
	Using column cycle and incomplete column cycle switching instead of
	row cycle switching we can bound \eref{e:term} when $t_i$ is in the
	$m$-th column.
	
	\begin{lem}\label{l:setcolcycs}
		Let $i=m^2-m+j$ where $j\in[m]$
		and let $\L$ be a random $k \times n$ Latin rectangle. Then
		\[
		\Pr(\L \in \Delta_i \mid  \L \in \Delta_{i-1}) \leq \frac j{n+j-m}.
		\]
	\end{lem}
	
	\begin{proof}
		The proof is similar to the proof of \lref{l:setrowcycs}. We will again denote the set of $k \times n$ Latin rectangles containing $P$ by $\mathcal{L}_P$ for $P \in \Gamma(T_i)$. 
		
		Let $c' \in [n] \setminus [m]$. Note that if $L \in \Delta_i$ then $L_{j, m} \in [m]$ and $L_{j, c'} \notin [m]$. Let $P \in \Gamma(T_i)$. By \lref{l:prelim2} there are at least $|\mathcal{L}_P|/j$ rectangles $L \in \mathcal{L}_P$ such that switching $L$ on $\sigma(m, c', j)$ yields a rectangle in $\Delta_{i-1} \setminus \Delta_i$. Thus the number of switches from a rectangle in $\Delta_i$ to one in $\Delta_{i-1} \setminus \Delta_i$ is at least $(n-m)|\Delta_i|/j$.
		
		Reversing the switching, consider $L' \in \Delta_{i-1} \setminus \Delta_i$. Note that $L'_{j, m} \notin [m]$ and there is a unique $c' \in [n]\setminus [m]$ such that $L'_{j, c'} \in [m]$. Therefore there is at most one switch from $L'$ to a rectangle in $\Delta_i$, namely $\sigma(m, c', j)$.
		
		We have thus shown that $(n-m)|\Delta_i|/j \leq |\Delta_{i-1} \setminus \Delta_i|$ and so
		\[
		\Pr(\L \in \Delta_i \mid  \L \in \Delta_{i-1}) = \frac{|\Delta_i|}{|\Delta_i|+|\Delta_{i-1} \setminus \Delta_i|} \leq \frac j{n+j-m},
		\]
		as required.	
	\end{proof}
	
	We are now ready to prove \tref{t:subsqbound}.
	
	\begin{proof}[Proof of \tref{t:subsqbound}]
		As previously mentioned, the probability that $\L[R, C]$ is a Latin square of order $m$ with symbol set $S$ is equal to the probability that $\L \in \Delta_{m^2}$. By \eref{e:cr} combined with \lref{l:setnrc}, \lref{l:setrowcycs} and \lref{l:setcolcycs} we obtain the bound
		\[
		\begin{aligned}
			\Pr(\L \in \Delta_{m^2}) &\leq
			\Bigg(\prod_{j=1}^m\frac{j}{n+j-m}\Bigg)^2\Bigg(\prod_{j=2}^{m-1}\frac{j}{n+j-m}\Bigg)\Bigg(\prod_{j=2}^{m-1}\frac{j}{k+j-m}\Bigg) \\
			&=\frac{nk(n+1-m)(k+1-m)}{m^2\binom{n}{m}^3\binom{k}{m}}
		\end{aligned}
		\]
		as required.
	\end{proof}

	\subsection{Bounding probability of small subsquares}\label{ss:subsq4bound}
	
	The aim of this subsection is to prove the following theorem,
	which will enable us to bound the expected number of small
	subsquares. Our bound improves the bound in the previous
	subsection when $2<m=O(1)$, although it is worse when $m$ grows
	at least logarithmically in~$n$. The new bound will be used
	when $m\in\{4,5\}$ to get the accuracy required for \tref{t:main}.
	
	\begin{thm}\label{t:subsq4boundgen}
		Let $R \subseteq [k]$ be of cardinality $m$, let $C$ and $S$ be
		subsets of $[n]$ of cardinality $m$, and let $\L$ be a random
		$k\times n$ Latin rectangle where $k>m>2$.
		The probability that $\L[R, C]$ is a Latin
		square of order $m$ with symbol set $S$ is at most
		\[
		\frac{2^{m-1}m^{m^2+2m-5}}{(k-m)^{5m/2-6}}\left(\frac{1+O(m/n)}n\right)^{3m-2}
		\]
		if $m$ is even and at most
		\[
		\frac{2^{m-2}m^{m^2+2m-5}}{(k-m)^{(5m-13)/2}}\left(\frac{1+O(m/n)}n\right)^{3m-2}
		\]
		if $m$ is odd. If $k=m>2$ then the same bounds hold with $k-m$ replaced by $1$.
	\end{thm}
	
	As in \sref{ss:subsqbound}, we may assume that $R=C=S=[m]$. 
	Also, if $M$ and $M'$ are any two Latin squares of order $m$, then
	replacing $M$ by $M'$ gives an easy bijection from Latin rectangles
	containing $M$ as a subsquare to Latin rectangles containing $M'$. Hence,
	if $\L$ is a random $k \times n$ Latin rectangle then
	\begin{equation}\label{e:swapM}
		\Pr(\L[[m], [m]] = M) = \Pr(\L[[m], [m]] = M').
	\end{equation}
	Thus without loss of
	generality we may assume that $\L[[m],[m]]=M$ is a specific Latin square
	of order $m$, which we will choose later. Ultimately, we will multiply
	by $m^{m^2}$, which is a trivial upper bound on the number of Latin
	squares of order $m$, to obtain \tref{t:subsq4boundgen}.	
	The strategy is similar to the proof of \tref{t:subsqbound}.
	Recalling \eref{e:Gammastar}, we let
	$P$ and $P'$ be elements of $\Gamma^*([m]^2)$ such that $P \subseteq
	P'$ and $P'\setminus P=\{(r, c, s)\}$. We will give non-trivial upper
	bounds on
	\begin{equation}\label{e:term2}
		\Pr(\L \supseteq P' \mid  \L \supseteq P)	
	\end{equation}
	when $P$ and $P'$ satisfy certain conditions. 
	
	First, we bound \eref{e:term2} when $(r, c, s)$ is the only occurrence
	of row $r$, the only occurrence of column $c$, or the only occurrence
	of symbol $s$ in $P'$.
	
	\begin{lem}\label{l:newrcs}
		Let $\{P, P'\} \subseteq \Gamma^*([m]^2)$ be such that $P \subseteq
		P'$ and $P' \setminus P = \{(r, c, s)\}$. Let $\L$ be a random
		$k\times n$ Latin rectangle, where $k>m$. Suppose that $P$ has no
		entry in row $r$. Then
		\begin{equation*}
			\Pr(\L \supseteq P' \mid  \L \supseteq P) \leq \frac1{k-m}.
		\end{equation*}
		If $P$ has no entry in column $c$ or no entry with symbol $s$ then
		\begin{equation*}
			\Pr(\L \supseteq P' \mid  \L \supseteq P) \leq \frac{1+O(m/n)}n.
		\end{equation*}
	\end{lem}
	
	\begin{proof}
		First consider the case where $P$ has no entry in row $r$. Let $Y$
		be the set of $k \times n$ Latin rectangles containing $P$ and let $X
		\subseteq Y$ be the set of rectangles containing $P'$. We use row cycle
		switching to find bounds on the relative sizes of $X$ and $Y\setminus X$.
		
		Let $L \in X$. For every row $r' \in [k]\setminus [m]$, switching
		on $\rho(r, r', c)$ yields a rectangle in $Y \setminus X$. This gives
		exactly $k-m$ switches from $L$ to a rectangle in $Y \setminus X$.
		
		Now, consider $L' \in Y \setminus X$ and note that there is at most
		one $r' \in [k]$ such that $L'_{r',c}=s$. Thus there is at most one
		row cycle which can be switched back to change $L'$ into some
		rectangle in $X$.  Thus at most one switch counts from $L'$ to a
		rectangle in $X$.
		
		We have deduced that $(k-m)|X| \leq |Y \setminus X| \leq |Y|$. Thus
		\[
		\Pr(\L \supseteq P' \mid  \L \supseteq P) = \frac{|X|}{|Y|} \leq \frac1{k-m}.
		\]
		
		To deal with the case where $P$ has no entry in column $c$ or has no
		entry with symbol $s$ we use analogous arguments with column cycle
		and incomplete column cycle switching or symbol cycle and incomplete
		column cycle switching, respectively. In each case there are
		exactly $n-m$ switches from $L$ to a rectangle in $Y \setminus X$ and
		hence
		\[
		\Pr(\L \supseteq P' \mid  \L \supseteq P) \leq \frac1{n-m}=\frac{1+O(m/n)}n.\qedhere
		\]
	\end{proof}
	
	We now bound \eref{e:term2} when $P'$ contains a cell $(r,c)$
	satisfying the conditions of \lref{l:prelim}.
	
	\begin{lem}\label{l:colcycs}
		Let $\{P, P'\} \subseteq \Gamma^*([m]^2)$ be such that $P \subseteq
		P'$ and $P' \setminus P = \{(r, c, s)\}$. Let $C'$ be the set of
		columns $c'$ for which the cell $(r, c')$ of $P'$ is
		non-empty. Suppose further that for every $c' \in C'\setminus\{c\}$,
		the entries of $P'$ in columns $c$ and $c'$ form a
		union of column cycles. Let $\L$ be a random $k \times n$ Latin
		rectangle where $k>m$. Then
		\[
		\Pr(\L \supseteq P' \mid  \L \supseteq P) \leq \frac{|C'|}{k-m}.
		\]
	\end{lem}
	
	\begin{proof}
		Let $Y$ be the set of $k \times n$ Latin rectangles containing $P$ and
		let $X \subseteq Y$ be the set of rectangles containing $P'$.
		
		Let $r' \in [k]\setminus [m]$. By \lref{l:prelim} there are 
		$|X|/|C'|$ rectangles in $X$ for which switching on $\rho(r, r', c)$
		yields a rectangle in $Y \setminus X$. Thus there are 
		$(k-m)|X|/|C'|$ switches from $X$ to $Y \setminus X$.
		
		Now let $L' \in Y \setminus X$. The same argument as given in the
		proof of \lref{l:newrcs} tells us that there is at most one row cycle
		that can be switched from $L'$ to reach a rectangle in $X$.
		
		We have shown that $(k-m)|X|/|C'| \leq |Y \setminus X| \leq |Y|$. Thus
		\[
		\Pr(\L \supseteq P' \mid  \L \supseteq P) = \frac{|X|}{|Y|} \leq \frac{|C'|}{k-m},
		\]
		as required.
	\end{proof}
	
	By using column cycle and incomplete column cycle switching rather
	than row cycle switching and employing \lref{l:prelim2} we can prove
	the following lemma.
	
	\begin{lem}\label{l:rowcycs}
		Let $\{P, P'\} \subseteq \Gamma^*([m]^2)$ be such that $P \subseteq
		P'$ and $P' \setminus P = \{(r, c, s)\}$. Let $R'$ be the set of
		rows $r'$ for which the cell $(r', c)$ of $P'$ is non-empty. Suppose
		further that for every $r' \in R' \setminus \{r\}$, the entries of
		$P'$ in rows $r$ and $r'$ form a union of row cycles. Let $\L$ be a
		random $k \times n$ Latin rectangle. Then
		\[
		\Pr(\L \supseteq P' \mid  \L \supseteq P) \leq \frac{|R'|}n(1+O(m/n)).
		\]
	\end{lem}
	
	Our next task is to choose the Latin square $M$ of order $m$. Our
	choice of $M$ allows us to apply \lref{l:newrcs}, \lref{l:colcycs} or
	\lref{l:rowcycs} many times to give non-trivial bounds on
	\eref{e:term2}. If $m$ is even then define $M$ to be a Latin square of
	order $m$ with the following properties:
	\begin{itemize}
		\item $M_{1, i} = M_{i, 1} = i$ for every $i \in [m]$,
		\item $M_{2, i} = M_{i, 2} = i+1$ for every odd $i \in [m]$,
		\item $M_{2, i} = M_{i, 2} = i-1$ for every even $i \in [m]$.
	\end{itemize}
	If $m$ is odd then define $M$ to be a Latin square of order $m$ satisfying:
	\begin{itemize}
		\item $M_{1, i} = M_{i, 1} = i$ for every $i \in [m]$, 
		\item $M_{2, i} = M_{i, 2} = i+1$ for every odd $i \in [m-3]$,
		\item $M_{2, i} = M_{i, 2} = i-1$ for every even $i \in [m-3]$,
		\item $(M_{2, m-2}, M_{2, m-1}, M_{2, m}) = (M_{m-2, 2}, M_{m-1, 2}, M_{m, 2}) = (m-1, m, m-2)$.
	\end{itemize}
	Such a square $M$ exists
	by~\cite[Theorem~$1.5$]{2rows2cols}. See \fref{f:LS45} 
	for examples of such squares when $m=4$ and $m=5$. These are the only
	two orders for which \tref{t:subsq4boundgen} will be used
	in the proof of \tref{t:main}.
	
	\begin{figure}
		\[
		\begin{minipage}{.4\textwidth}
			\begin{equation*}
				\left[\;
				\begin{matrix}
					\con1&\con2&\con3&\con4\\
					\con2&\cth1&\con4&\cth3\\
					\con3&4&1&\cth2\\
					\con4&\ctw3&\ctw2&\cth1
				\end{matrix}
				\;\right]
			\end{equation*}
		\end{minipage}
		\begin{minipage}{.4\textwidth}
			\begin{equation*}
				\left[\;
				\begin{matrix}
					\con1&\con2&\con3&\con4&\con5\\
					\con2&\cth1&\con4&\con5&\cth3\\
					\con3&4&5&1&\cth2\\
					\con4&5&2&3&\cth1\\
					\con5&\ctw3&\ctw1&\ctw2&\cth4
				\end{matrix}
				\;
				\right]
			\end{equation*}
		\end{minipage}
		\]
		\caption{\label{f:LS45}The Latin squares $M$ when
			$m\in\{4,5\}$. Colour coding
			$\;\begin{matrix}\con&,&\ctw&,&\cth\end{matrix}\;$ indicates cells
			to which Lemmas \ref{l:newrcs}, \ref{l:colcycs} and
			\ref{l:rowcycs}, respectively, will be applied in the proof of
			\tref{t:subsq4boundgen}.}
	\end{figure}
	
	We are now ready to prove \tref{t:subsq4boundgen}.
	
	\begin{proof}[Proof of \tref{t:subsq4boundgen}]
		First suppose that $m$ is even and $k>m$. We assume that $R=C=S=[m]$ and that $M$ is the Latin square of order $m$ defined above. Let $P$ be the $k \times n$ partial Latin rectangle such that $P[[m], [m]] = M$ and every cell of $P$ outside of $[m]^2$ is empty. We define a family of partial Latin rectangles $P_i \subseteq P$ as follows. Firstly, $P_0 = \emptyset$. Then for $i \in [m^2]$ we define $P_i$ by adding an entry of $P$ to $P_{i-1}$. 
		
		For $i \in [2m]$ we add the entry of $P$ in cell $(1, (i+1)/2)$ if $i$ is odd and we add the entry in cell $(2, i/2)$ if $i$ is even. If $i$ is odd then
		\begin{equation}\label{e:bnd1}
			\Pr(\L \supseteq P_i \mid  \L \supseteq P_{i-1}) \leq (1+O(m/n))/n
		\end{equation}
		by \lref{l:newrcs}, since there is only one non-empty cell of $P_i$ in column $(i+1)/2$. Similarly, if $i \equiv 2 \bmod 4$ then
		\eref{e:bnd1} holds
		by \lref{l:newrcs}, since the entry $(2, i/2, i/2+1)$ is the only entry of $P_i$ with symbol $i/2+1$. If $i \equiv 0 \bmod 4$ then
		\begin{equation}\label{e:bnd2}
			\Pr(\L \supseteq P_i \mid  \L \supseteq P_{i-1}) \leq 2(1+O(m/n))/n,
		\end{equation}
		by \lref{l:rowcycs}. 
		
		For $i \in \{2m+1, \ldots, 4m-4\}$ we add the entry of $P$ in cell
		$((i-2m+5)/2, 1)$ if $i$ is odd and we add the entry in cell
		$((i-2m+4)/2, 2)$ if $i$ is even. If $i$ is odd then 
		\begin{equation}\label{e:bndk1}
			\Pr(\L \supseteq P_i \mid  \L \supseteq P_{i-1}) \leq 1/(k-m)
		\end{equation}
		holds by \lref{l:newrcs}, since there is only one
		non-empty cell of $P_i$ in row $(i-2m+5)/2$. If
		$i \equiv 0 \bmod 4$ then 
		\begin{equation}\label{e:bndk2}
			\Pr(\L \supseteq P_i \mid  \L \supseteq P_{i-1}) \leq 2/(k-m)
		\end{equation}
		holds by \lref{l:colcycs}. If $i \equiv 2 \bmod 4$
		we simply use $\Pr(\L \supseteq P_i \mid \L \supseteq P_{i-1})\le1$.
		
		For $i \in \{4m-3, \ldots, m^2\}$ we add the entry of $P$ in cell $(3+\lfloor (i-4m+3)/(m-2) \rfloor, 3+((i-4m+3) \bmod m-2))$ (this is the natural ordering for the cells in $\{3, 4, \ldots, m\}^2$). If $i \equiv 4 \bmod m-2$ (in which case the added entry is in column $m$) then
		\begin{equation}\label{e:bnd3}
			\Pr(\L \supseteq P_i \mid  \L \supseteq P_{i-1}) \leq m(1+O(m/n))/n,
		\end{equation}
		by \lref{l:rowcycs}. If $i \geq m^2-m+3$ (in which case the added entry is in
		row $m$) then
		\begin{equation}\label{e:bndk3}
			\Pr(\L \supseteq P_i \mid  \L \supseteq P_{i-1}) \leq m/(k-m)
		\end{equation}
		by \lref{l:colcycs}.  For all other values of $i$ we
		use $\Pr(\L \supseteq P_i \mid \L \supseteq P_{i-1})\le1$.
		
		Overall, we use \eref{e:bnd1} for $3m/2$ values of $i$, corresponding
		to the $m$ entries in the first row and $m/2$ entries in the second
		row. We use \eref{e:bnd2} for $m/2$ values of $i$, corresponding to
		entries in the second row. We use \eref{e:bndk1} for $m-2$ values of
		$i$, corresponding to entries in the first column. We use
		\eref{e:bndk2} for $m/2-1$ values of $i$, corresponding to entries in
		the second column. We use \eref{e:bnd3} for $m-2$ values of $i$,
		corresponding to entries in the $m$-th column. We use \eref{e:bndk3}
		for $m-3$ values of $i$, corresponding to entries in the $m$-th row.
		Hence, by the chain rule of probability,
		\begin{align}
			\Pr(\L \supseteq P) &= \prod_{i=1}^{m^2} \Pr(\L \supseteq P_i \mid  \L \supseteq P_{i-1}) \nonumber\\
			&\leq \left(\frac{1+O(m/n)}n\right)^{3m/2}\left(\frac{2(1+O(m/n))}n\right)^{m/2}\left(\frac1{k-m}\right)^{m-2} \\
			&\hspace{5mm}\cdot\left(\frac{2}{k-m}\right)^{m/2-1}\left(\frac{m(1+O(m/n))}n\right)^{m-2}\left(\frac{m}{k-m}\right)^{m-3} \nonumber\\
			&= \frac{2^{m-1}m^{2m-5}}{(k-m)^{5m/2-6}}\left(\frac{1+O(m/n)}n\right)^{3m-2}.\label{e:evencase}
		\end{align}
		If $m$ is even and $k=m$ then the same bounds hold except we use the trivial
		bound in place of \eref{e:bndk1}, \eref{e:bndk2} and \eref{e:bndk3}.

		The argument when $m$ is odd is similar. We use \eref{e:bnd1} for
		$(3m+1)/2$ entries, namely $m$ entries in the first row and $(m+1)/2$
		entries in the second row. We use \eref{e:bnd2} for $(m-1)/2$ entries
		in the second row. If $k>m$ then we use \eref{e:bndk1} for $m-2$ entries in the
		first column and \eref{e:bndk2} for $(m-3)/2$ entries in the
		second column. We use \eref{e:bnd3} for $m-2$ entries in the $m$-th
		column and if $k>m$ then we use \eref{e:bndk3} for $m-3$ entries in the $m$-th
		row. As a consequence,
		\begin{equation}\label{e:oddcase}
			\Pr(\L \supseteq P) \leq \frac{2^{m-2}m^{2m-5}}{(k-m)^{(5m-13)/2}}\left(\frac{1+O(m/n)}n\right)^{3m-2},
		\end{equation}
		with $k-m$ replaced by $1$ in the case when $k=m$.
		As foreshadowed after \eref{e:swapM}, we now
		multiply by $m^{m^2}$ to give the claimed upper bound
		on $\L[R, C]$ being any subsquare on the symbols~$[m]$.
	\end{proof}
	
	\subsection{Proof of \tref{t:main}}\label{ss:combine}
	
	We are now ready to prove our main theorem.
	
	\begin{proof}[Proof of \tref{t:main}]
		Let $m$ be an integer satisfying $4 \leq m \leq \min\{k, n/2\}$.
		There are $\binom{k}{m}{}$ choices for the set $R$ and there are
		$\binom{n}{m}{}^2$ choices for the sets $C$ and $S$ in
		\tref{t:subsqbound} and \tref{t:subsq4boundgen}. From
		\tref{t:subsqbound} and $\binom{n}{m}\geq (n/m)^m$, we obtain
		\begin{align}
			\E_m(k, n) &\leq \frac{kn(k+1-m)(n+1-m)}{m^2\binom{n}{m}}
			\leq\frac{m^{m-2}}{n^{m-4}}=\frac{m^4}{n^2}\Big(\frac mn\Big)^{m-6}\le\frac{m^4}{n^22^{m-6}}=O(n^{-2}), \label{e:p2}
		\end{align}
		provided $m\ge6$.
		
		For $m\in\{4,5\}$ and $k>m$ we use \tref{t:subsq4boundgen} to deduce that
		\begin{align*}
			\E_4(k,n) &\leq \binom{k}{4}\binom{n}{4}^2O(1)\left(\frac{1+O(1/n)}n\right)^{10}\left(\frac{1}{k-4}\right)^{4}=O(n^{-2}),
		\end{align*}
		and
		\begin{align*}
			\E_5(k,n) &\leq \binom{k}{5}\binom{n}{5}^2O(1)\left(\frac{1+O(1/n)}n\right)^{13}\left(\frac{1}{k-5}\right)^{6}
			= O(n^{-3}).
		\end{align*}
		When $k=m\in\{4,5\}$ all terms involving $k$ can be omitted, with the
		conclusion unchanged.  Combining with \eref{e:p2}, we have the
		result.
	\end{proof}
	
	\section{Subsquares of order $2$ or $3$}\label{s:subsq3}
	
	From \tref{t:main} we know that $\E_m(k,n)$ is asymptotically $0$ for
	all $m\ge4$. In this last section we consider subsquares of order $m<4$.
	McKay and Wanless~\cite{manysubsq} conjectured that
	$\E_3(n,n)=1/18+o(1)$. Kwan, Sah and Sawhney \cite{KSS} conjectured
	further that the distribution of the number of $3\times3$ subsquares
	in a random Latin square is asymptotically Poisson with this mean. If
	these conjectures are true, then $\lim_{n \to \infty} \E_m(n, n)$ is
	both positive and finite only when $m=3$.  Divoux, Kelly, Kennedy and
	Sidhu~\cite{subsqrandom} generalised McKay and Wanless' conjecture by
	suggesting that $\E_3(k, n) = (1+o(1))\binom{k}{3}/(3n^3)$.  Although
	we are not able to resolve the situation completely, we are able to
	prove that $\E_3(k, n)$ is bounded. The proof follows the derivation
	of \eref{e:oddcase} closely, but when $m=3$ that result was slightly
	conservative when counting entries in the second column to which
	\lref{l:colcycs} can be applied.
	
	\begin{thm} For $3\le k\le n$,
		\[
		\E_3(k, n) \leq \frac{2k}{3n}(1+O(1/k)).
		\]
	\end{thm}
	
	\begin{proof} 
		Let $M$ denote the Latin square
		\[
		\left[\;
		\begin{matrix}
			\con1&\con2&\con3\\
			\con2&\con3&\cth1\\
			\con3&\ctw1&\cth2
		\end{matrix}
		\;
		\right],
		\]
		with the same colouring convention as used in \fref{f:LS45}.  Let
		$P$ be the partial $k \times n$ Latin rectangle such that
		$P[[3],[3]]=M$ and every cell of $P$ outside of $[3]^2$ is empty. We
		first bound the probability that a random $k \times n$ Latin
		rectangle $\L$ contains $P$ in the case when $k>3$. We define a
		family of partial Latin rectangles $P_i\subseteq P$ as follows. Let
		$P_0 = \emptyset$ and for $i\in[9]$ let $P_i$ be obtained from
		$P_{i-1}$ by adding the entry of $P$ in the cell given by the $i$-th
		element of the tuple $((1, 1), (2, 1), (1, 2), (2, 2), (1, 3), (2,
		3), (3, 1), (3, 2), (3, 3))$. For $i \in [5]$ we have that
		\eref{e:bnd1} holds by \lref{l:newrcs}. For $i = 6$ we have that
		\eref{e:bnd2} holds by \lref{l:rowcycs}. For $i = 7$ we have that
		\eref{e:bndk1} holds by \lref{l:newrcs}. For $i = 8$ we have that
		\eref{e:bndk2} holds by \lref{l:colcycs}. Finally, we see that
		$\Pr(\L \supseteq P_9 \mid \L \supseteq P_{8}) \leq 3(1+O(1/n))/n$
		by \lref{l:rowcycs}. The chain rule of probability implies that
		\[
		\Pr(\L \supseteq P) \leq 12\left(\frac{1+O(1/n)}n\right)^7\left(\frac{1}{k-3}\right)^2 = \frac{12}{n^7(k-3)^2}(1+O(1/n)).
		\]
		Since there are $12$ Latin squares of order $3$ we obtain
		\[
		\E_3(k,n) \leq \binom{k}{3}\binom{n}{3}^2\frac{144}{n^7(k-3)^2}(1+O(1/n)) = \frac{2k}{3n}(1+O(1/k)),
		\]
		as required. If $k=3$ then we instead get
		\[
		\E_3(k,n) \leq \binom{n}{3}^2\frac{72}{n^7}(1+O(1/n)) = \frac{2}{n}(1+O(1/n)).
		\qedhere  \]
	\end{proof}
	
	Finally, we turn to the case $m=2$. Recall from \tref{t:intercs} that
	$\mathbb{E}_2(n,n) =
	(1+o(1))n^2/4$. Also,~\cite[Corollary~$1.7$]{subsqrandom} tells us
	that $\mathbb{E}_2(k,n)=(1/2+o(1))\binom{k}{2}$ whenever
	$k\leq(1/2-o(1))n$. We extend this result to all $k \leq n$. We first
	need the following lemma.
	
	\begin{lem}\label{l:intercbound}
		Let $\mathbf{L}$ be a random $k \times n$ Latin rectangle and let
		$\{i, j\} \subseteq [k]$ with $i \neq j$. The probability that
		$\L[\{i, j\}, [n]]$ contains at least $t$ intercalates is at most
		$\exp(-\Omega(t\log t))$.
	\end{lem}
	
	\begin{proof} 
		Let $P$ be the partial $k \times n$ Latin rectangle which is empty except that:
		\begin{itemize}
		\item $P_{i, \ell} = \ell$ for every $\ell \in [2t]$,
		\item $P_{j, \ell} = \ell+1$ for every odd $\ell \in [2t]$, and
		\item $P_{j, \ell} = \ell-1$ for every even $\ell \in [2t]$.
		\end{itemize}
		First suppose that $t\le n/4$, so that we can use the argument behind
		\eref{e:bnd1} and \eref{e:bnd2} to obtain
		\[
		\Pr(\mathbf{L} \supseteq P) \leq 2^tn^{-4t}(1+O(t/n))^{4t}.
		\]
		For $t$ intercalates in general position in $\mathbf{L}[\{i, j\}, [n]]$
		there are $\binom{n}{2t}$ choices for the symbols, $(2t)!/(2^tt!)$ ways
		to pair them, and then $n!/(n-2t)!$ ways to allocate the columns.
		Thus the probability that $\L[\{i,j\},[n]]$ contains at least $t$
		intercalates is at most
		\[
		\left(\frac{n!}{(n-2t)!}\right)^2\frac1{2^tt!}2^tn^{-4t}(1+O(t/n))^{4t}
		= \exp(-\Omega(t\log t)),
		\]
		as required. If $n/4<t\le n/2$ then $\Omega(t\log t)=\Omega((n/4)\log(n/4))$,
		so we can derive the same bound simply
		by considering the first $n/4$ intercalates in $\mathbf{L}[\{i,j\},[n]]$.
	\end{proof}

\begin{thm}
  For $2 \leq k \leq n$, 
  \[
  \mathbb{E}_2(k,n)=\frac12\binom{k}{2}(1+o(1)).
  \]
\end{thm}

\begin{proof}
  Since~\cite[Corollary $1.7$]{subsqrandom} proves the claim when
  $k\leq (1/2-o(1))n$, we may assume that $k = \Theta(n)$, which means
  that arguments that have previously been used for counting intercalates in
  Latin squares apply with only minor changes. Let $k'$ be a
  integer function of $n$ with $0\leq k'\leq k$. For a $k' \times n$
  Latin rectangle $T$ let $\mathcal{R}(T)$ be the set of $k \times n$
  Latin rectangles $L$ such that $L[[k'], [n]] = T$.
  By~\cite[Proposition~$4$]{manysubsq}, if $T$ and $T'$ are $k'\times n$
  Latin rectangles, then
  \begin{equation}\label{e:ext}
    \frac{|\mathcal{R}(T)|}{|\mathcal{R}(T')|} \leq \exp(O(n\log^2n)).
  \end{equation}
  
  Let $\mathbf{N}$ denote the number of intercalates in a random $k
  \times n$ Latin rectangle. Modifying the proof
  of~\cite[Theorem $4$]{KS}, by using \lref{l:intercbound}
  instead of~\cite[Theorem $3$]{KS} and using \eref{e:ext}, gives
  \begin{equation}\label{e:ub}
    \Pr(\mathbf{N} \leq (1-\delta)k^2/4) \leq \exp(-\Omega(n^{1/2} \log^{-1} n)),
  \end{equation}
  for fixed $\delta \in (0, 1]$. Similarly, a simple adaptation of the proof
    of~\cite[Theorem $2.1$]{KSSS} using \eref{e:ext} gives
    \begin{equation}\label{e:lb}
      \Pr(\mathbf{N} \geq (1+\delta)k^2/4) \leq \exp(-\Omega(n^{4/3}\log n)),
    \end{equation}
    for fixed $\delta > 0$.
    Combining \eref{e:ub} and \eref{e:lb} proves the theorem.
\end{proof}
	
\printbibliography
	
\end{document}